\documentclass[12pt]{article}

\usepackage[dvips]{epsfig} 
\usepackage{amssymb,amsmath,amsthm,mathrsfs} 
\usepackage{graphicx}
\usepackage{lineno}
\usepackage{multirow}

\usepackage[authoryear,sort&compress]{natbib}
\usepackage{url}
\usepackage{enumerate}
\usepackage{colordvi,color,pspicture}
\usepackage{psfrag}

\newcommand{\E}{\mathbf{E}}
\newcommand{\Var}{\mathbf{Var}}
\newcommand{\Cov}{\mathbf{Cov}}

\newcommand{\s}{\sigma}

\newtheorem{thm}{Theorem}[section]

\newtheorem{lemma}[thm]{Lemma}

\newtheorem{remark}[thm]{Remark}

\theoremstyle{definition}
\newtheorem{definition}[thm]{Definition}

\begin{document}
\title{On the Number of Reflexive and Shared Nearest Neighbor Pairs in One-Dimensional Uniform Data}
\author{
Selim Bahad{\i}r \& Elvan Ceyhan\\
Department of Mathematics, Ko\c{c} University,\\
Sar{\i}yer, 34450, Istanbul, Turkey.
}
\date{\today}
\maketitle

\pagenumbering{arabic} \setcounter{page}{1}

\begin{abstract}
\noindent
For a random sample of points in $\mathbb{R}$,
we consider the number of pairs whose members are nearest neighbors (NN) to each other and
the number of pairs sharing a common NN.
The first type of pairs are called reflexive NNs whereas latter type of pairs are called shared NNs.
In this article,
we consider the case where the random sample of size $n$ is from the uniform distribution on an interval.
We denote the number of reflexive NN pairs and the number of shared NN pairs in the sample as $R_n$ and $Q_n$, respectively.
We derive the exact forms of the expected value and the variance for both $R_n$ and $Q_n$,
and derive a recurrence relation for $R_n$
which may also be used to compute the exact probability mass function of $R_n$.
Our approach is a novel method for finding the pmf of $R_n$ and agrees with the results in literature.
We also present SLLN and CLT results for both $R_n$ and $Q_n$ as $n$ goes to infinity.
\end{abstract}


\section{Introduction}
\label{sec:intro}
The nearest neighbor (NN) relations and their properties have been extensively studied in various fields,
such as probability and statistics (\cite{bickel:1983}), computer science (\cite{yao:1997},
and ecology (\cite{clark:1954}).
Based on the NN relations,
NN graphs and digraphs are constructed and related graph quantities/invariants are widely studied
(\cite{yao:1997}, \cite{penrose:2001}, and \cite{kozakova:2006}).
We consider NN digraphs and quantities based on their arcs (i.e., directed edges).
In a NN digraph, $D=(V,A)$, the vertices are data points in $\mathbb{R}^d$,
and there is an arc from vertex $u$ to vertex $v$ (i.e., $(u,v) \in A$) if $v$ is a NN of $u$.
We call a pair of vertices $u,v$ as a \emph{reflexive NN pair},
if $v$ is a NN of $u$ and vice versa (i.e.,$ \{ (u,v),(v,u)\} \subset A$) (\cite{clark:1955} and \cite{cox:1981}).
If  both $(u,w)$ and $(v,w)$ are in $A$ for some $w\in V$ (i.e., $u$ and $v$ share a NN), then $(u,w)\& (v,w)$ are called \emph{shared NNs}.
Notice that altough $w$ is the shared vertex, arcs $(u,w)$ and $(v,w)$ are called shared NNs in literature (see, \cite{dixon:1994}).
The vertices in a reflexive NN pair are also called isolated NNs (\cite{pickard:1982}),
mutual NNs (\cite{schiling:1986}) or biroot (\cite{yao:1997}).

The NN digraph is also referred as the NN graph in literature (e.g., \cite{yao:1997});
but, since the NN relation is not symmetric,
we opt to use ``NN digraph" which reflects this asymmetry.
Also, the underlying graph of a NN digraph
(an underlying graph of a digraph is obtained by replacing each arc with an (undirected) edge,
disallowing multiple edges between two vertices (\cite{chartrand:1996})) is sometimes referred to as the NN graph
(see, e.g., \cite{friedman:1983}, \cite{penrose:2001}).
Since in any (undirected) graph,
the relation defining the edges is symmetric (i.e., each edge is symmetric),
reflexivity is not an interesting property for undirected graphs.

Number of reflexive and shared NN pairs in a NN digraph is of importance in various fields.
For example,
in spatial data analysis,
the distributions of the tests based on nearest neighbor contingency tables depend on these two quantities (\cite{dixon:1994} and \cite{ceyhan:cell2008}),
when the underlying pattern of the points is from a spatial distribution
(e.g., from homogeneous Poisson process (HPP) or a binomial process).
Moreover, neighbor sharing type quantities such as $Q_n$ are also of interest for the problem of estimating the intrinsic dimension of a data set (see, \cite{brito:2013}).

In our analysis,
we consider the special case of $d=1$ (i.e., one dimensional data),
and study the case when the random sample of size $n$ is obtained from uniform distribution over an interval.
We denote the total number of reflexive and shared NN pairs in the corresponding sample as $R_n$ and $Q_n$, respectively.
The quantity $R_n$ could be of interest for inferential purposes as well,
since it is a measure of mutual (symmetric) spatial dependence between points,
which might indicate a special and/or stronger form of clustering of data points.
For instance, a simple test based on the proportion of the number of
reflexive pairs to the sample size was presented by \cite{dacey:1960}
to interpret the degree of regularity or clustering of the locations of towns alongside a river.
However, the methodology of \cite{dacey:1960} ignores the randomness (and hence uncertainty)
in the value of $R_n$ and hence is not reliable.
The exact distribution of $R_n$ can be computed for finite values of $n$ and hence,
would make possible the use of $R_n$ in exact inference
for testing such one-dimensional clustering.

NN relations, such as reflexivity and neighbor sharing, are studied by many authors.
\cite{enns:1999} provide $\E(R_n)=n/3$ for $n\geq 3$, $\Var (R_n)=2n/45$ for $n\geq 5$ and a recurrence relation giving
the exact pmf of $R_n$ for finite $n$,
whereas the results in \cite{schiling:1986} yield
$\E(Q_n)/ n \rightarrow 1/4$ as $n \rightarrow \infty$.
For the number of reflexive pairs, we approach to the problem in the same way as \cite{enns:1999},
but we drive the mean, the variance and the recurrence relation by a different approach.
Further, we obtain mean and variance of $Q_n$ and compute the asymptotic distribution of both $R_n$ and $Q_n$,
which are novel contributions of this article to the literature.
We provide preliminary results in Section \ref{sec:prelim} where
we convert our problems into random permutations by using interchangeability of uniform spacings.
We derive means and variances of $R_n$ and $Q_n$ together with a recurrence relation giving
the exact pmf of $R_n$ in Section \ref{sec:mean-var}.
The asymptotic results (such as SLLN and CLT) for $R_n$ and $Q_n$ are presented in Section \ref{sec:asy-res},
and discussion and conclusions are provided in Section \ref{sec:disc-conc}.

\section{Preliminaries}
\label{sec:prelim}
The number of reflexive pairs and the number of shared neighbors in the data is invariant under translation and scaling,
since both depend only on the ordering of the pairwise distances of the data points.
Therefore, without loss of generality, we may only consider the uniform distribution over the interval $(0,1)$ (denoted $U(0,1)$).

A NN of a point is one of the ``closest'' points
with respect to some distance or dissimilarity measure.
We will employ the usual Euclidean distance in our analysis.
Observe that under uniform distribution,
the Lebesgue measure of the set of points which have more than one NN
is zero and therefore we may assume that each point has a unique NN with probability 1.
In a sample of size $n$ from $U(0,1)$,
recall that a pair of points is called \emph{reflexive}, if each one is the NN of the other,
and we denote the total number of reflexive pairs as $R_n$, and
a pair of points in the sample is called \emph{shared NN}, if they have the same NN (sharing the NN) and we denote the total number of shared NNs as $Q_n$.

Let $\{ U_1,U_2,\dots , U_n \}$ be a random sample of size $n$ from the uniform distribution $U(0,1)$.
On the real line,
there is a nice ordering structure for the data which we exploit in our results.
Let $ U_{(1)}, U_{(2)},\dots , U_{(n)}$ be the order statistics of $\{ U_1,U_2,\dots , U_n \}$.
Denote the spacings between the order statistics as $D_i:=U_{(i+1)}-U_{(i)}$ for $1\leq i \leq n-1$ with $D_0:=U_{(1)}$.

\begin{lemma}\label{lem:obsrn}
For $n\geq 3$
$$R_n={\bf 1}_{ \{D_1<D_2 \} }+\sum_{i=2}^{n-2} {\bf 1}_{ \{ D_{i}<\min\{ D_{i+1}, D_{i-1}\} \} }+{\bf 1}_{ \{ D_{n-1}<D_{n-2} \} },$$
and for $n\geq 4$
$$Q_n={\bf 1}_{ \{D_2<D_3 \} }+\sum_{i=2}^{n-3} {\bf 1}_{ \{ D_{i}< D_{i-1}, D_{i+1}<D_{i+2} \} }+{\bf 1}_{ \{ D_{n-2}<D_{n-3} \} },$$
where ${\bf 1}_{A}$ is the indicator for the event $A$.
\end{lemma}
\begin{proof}
First observe that the NNs of $U_{(1)}$ and $U_{(n)}$ are always $U_{(2)}$ and $U_{(n-1)}$, respectively.
Therefore,
$\{ U_{(1)}, U_{(2)} \}$ is a reflexive pair if and only if $D_1<D_2$ and, similarly, $\{ U_{(n-1)}, U_{(n)} \}$ is reflexive if and only if $D_{n-1}<D_{n-2}$.
Also note that, for each $2\leq i \leq n-1$, the NN of $U_{(i)}$ is either $U_{(i-1)}$ (if $D_{i-1}<D_i$) or $U_{(i+1)}$ (if $D_i<D_{i-1}$).
Thus, for $2\leq i,j\leq n-1$,
the pair $\{U_{(i)}, U_{(j)}\}$ with $i<j$ is reflexive if and only if $j=i+1$ and $D_{i}$ is less than both $D_{i-1}$ and $D_{i+1}$.
So, we obtain the first identity in the Lemma \ref{lem:obsrn}.
For the representation of $Q_n$, in a similar manner,
one can easily see that $U_{(i)}$ and $U_{(j)}$ ($i<j$) have the same NN only if $j=i+2$ and the common NN is $U_{(i+1)}$,
and obtain the desired result.
\end{proof}

As the quantities $R_n$ and $Q_n$ depend on the ordering of the spacings,
we focus on the distribution of the spacings.
By elementary probability arguments (e.g., Jacobian density theorem) it follows that the joint density of the spacings $(D_0,D_1,\dots , D_{n-1})$ is
\begin{align}\label{eq:unispdens}
f_S(d_0,d_1,\dots, d_{n-1})=n! {\bf 1}_{ \{d_0+d_1+\cdots +d_{n-1}<1\}}  {\bf 1}_{ \{ \min \{d_0,d_1, \dots ,d_{n-1} \}>0\} }
\end{align}
with the understanding that $\{d_0,d_1, \dots, d_{n-1}>0\}=\{d_0>0, d_1>0, \dots, d_{n-1}>0\}$.
By \eqref{eq:unispdens} it is clear that the spacings $D_1,\dots ,D_{n-1}$ are interchangeable and hence
$P(D_{\sigma (1)}< \cdots < D_{\sigma (n-1)})=P(D_1<\cdots <D_{n-1})$ for any permutation $\sigma$ in $P_{n-1}$,
where $P_{n-1}$ is the permutation group on $\{1,2,\dots , n-1\}$.
In other words, every ordering of the spacings $D_1,\dots ,D_{n-1}$ is equally likely to occur.

Let $\sigma$ be chosen uniformly at random from $P_{n-1}$.
Define the events $A_1=\{\sigma(1)<\sigma(2)\}$, $A_{n-1}=\{\sigma(n-1)<\sigma(n-2)\}$,
$A_i=\{\sigma(i)<\sigma(i-1), \sigma(i)< \sigma(i+1) \}$ for all $2\leq i \leq n-2$,
and the events $B_1={\{\s(2)<\s(3)\}}$, $B_{n-2}={\{\s(n-2)<\s(n-3)\}}$ and
$B_i={\{\s(i)<\s(i-1), \s(i+1)<\s(i+2)\}}$ for each $2\leq i \leq n-3$.
Then, by Lemma \ref{lem:obsrn} and the interchangeability of the spacings we have
\begin{align}
R_n\stackrel{d}{=}\sum_{i=1}^{n-1} {\bf 1}_{A_i}  \text{ and } Q_n\stackrel{d}{=}\sum_{i=1}^{n-2} {\bf 1}_{B_i}, \label{eq:equiv}
\end{align}
where $\stackrel{d}{=}$ denotes equality in distribution.
Therefore, throughout of this paper, we consider $\sum_{i=1}^{n-1} {\bf 1}_{A_i}$ and $\sum_{i=1}^{n-2} {\bf 1}_{B_i}$
for the probabilistic results for $R_n$ and $Q_n$, respectively.

\section{Some Probabilistic Results for $R_n$ and $Q_n$}
\label{sec:mean-var}
In this section, we derive the means and variances of $R_n$ and $Q_n$, and present a recurrence relation for the exact distribution of $R_n$.
\subsection{Mean and Variance of $R_n$}
In a digraph $D$, a \emph{weakly connected component} is a maximal subdigraph of $D$ in which there is a path from every vertex to every other vertex in the underlying graph of $D$.
\cite{enns:1999} call a weakly connected component of a digraph as \emph{society} and examine the number of societies in a uniform data of size $n$ in one dimension.
By the simple observation that each society contains exactly one reflexive pair,
they convert the problem into the number of reflexive pairs and focus on the ranking of the spacings.
Considering the spacing with the largest length, they derive a recurrence relation and obtain
$\E(R_n)=n/3$ for $n\geq 3$ and $\Var (R_n)=2n/45$ for $n\geq 5$ by using generating functions.
We verify their results by following the idea in \cite{romik:2011}.

We obtain the mean and variance of $R_n$ by computing those of $\sum_{i=1}^{n-1} {\bf 1}_{A_i}$.
The random variable $\sum_{i=1}^{n-1} {\bf 1}_{A_i}$ is closely related to the length of the longest alternating subsequence in a random permutation
(see, e.g., \cite{romik:2011}, \cite{houdre:2010}).
For a sequence of pairwise distinct real numbers $x_1,\dots ,x_n$, a subsequence $x_{i_1}, \dots , x_{i_k}$ with
$1\leq i_1< \cdots <i_k \leq n$ is called \emph{alternating} if it satisfies
\begin{align*}
x_{i_1}>x_{i_2} < x_{i_3} > \cdots x_{i_k}.
\end{align*}
Note that there may be more than one alternating subsequence with the maximal length.
For instance, the sequence $6,4,1,3,5,2$ has seven longest alternating subsequences, particularly
$(6,1,3,2)$, $(6,1,5,2)$, $(6,4,5,2)$, $(6,3,5,2)$, $(4,1,3,2)$, $(4,1,5,2)$ and $(4,3,5,2)$.
Let the random variable $L_n^{as}$ be the maximal length of an alternating subsequence of $\tau(1),\dots , \tau(n)$, where $\tau$ is a uniformly random permutation from $P_n$.

Also, for $2\leq k \leq n-1$, $x_k$ is called \emph{local minimum} (resp. \emph{local maximum}) if $x_k<\min \{x_{k-1}, x_{k+1} \}$
(resp. $x_k>\max \{x_{k-1}, x_{k+1} \}$) (\cite{romik:2011}).
Note that the sum $\sum_{i=2}^{n-2} {\bf 1}_{A_i}$ is the number of local minima in $\sigma(1),\dots , \sigma(n-1)$,
where $\sigma$ is a uniformly random permutation from $P_{n-1}$.
\cite{romik:2011} shows that $L_n^{as}$ is equal to $1+\bf{1}_{ \{ \tau(1)>\tau(2)\}}$ plus the number of local minimums and local maximums in a random permutation, and provides $\E(L_n^{as})=2n/3+1/6$ and $\Var(L_n^{as})=8n/45-13/180$ (which are also computed in \cite{stanley:2008} and \cite{houdre:2010} in different ways).
Notice that the number of local minima and number of local maxima differ by at most one,
and hence the number of local minima is about half of $L_n^{as}$.
Therefore, we have $\E(R_n)/n\rightarrow 1/3$ and $\Var(R_n)/n \rightarrow 2/45$ as $n\rightarrow \infty$.
In fact, the limits $1/3$ and $2/45$ are actually attained for every $n\geq 5$.

\begin{thm}\label{thm:meanvarrn}
For a random sample of size $n$ from $U(0,1)$, the mean and the variance of the number of reflexive pairs, $R_n$, is $n/3$ (for $n\geq 3$) and $2n/45$ (for $n\geq 5$), respectively.
\end{thm}
\begin{proof}
By \eqref{eq:equiv} it suffices to derive the mean and the variance of $\sum_{i=1}^{n-1} {\bf 1}_{A_i}$.
We first compute the mean.
Clearly, $\E( {\bf 1}_{A_1})=P(A_1)=P(\s(2)<\s(1))=1/2$ and similarly by symmetry $\E( {\bf 1}_{A_{n-1}})=1/2$.
For $2\leq i \leq n-2$, we easily get $\E( {\bf 1}_{A_i})=P(A_i)=P(\s(i)<\min\{\s(i-1),\s(i+1)\})=1/3$.
Thus, for $n\geq 3$ we obtain
\[
\E(R_n)=\E \left(\sum_{i=1}^{n-1} {\bf 1}_{A_i} \right)= \sum_{i=1}^{n-1} \E({\bf 1}_{A_i}) =\frac{1}{2}+(n-3)\frac{1}{3} +\frac{1}{2}=\frac{n}{3}.
\]
For the variance of $R_n$, we derive the covariances of ${\bf 1}_{A_i}$'s given in the following matrix:
\[
\left( \Cov({\bf 1}_{A_i}, {\bf 1}_{A_j}) \right)_{i,j =1}^{n-1}=
\left( \begin{array}{cccccccccc}
\vspace{.2cm}
\frac{1}{4} & \frac{-1}{6} & \frac{1}{24} & 0 & 0 & 0 & 0 & 0 & \cdots & 0\\
\vspace{.2cm}
\frac{-1}{6} & \frac{2}{9} & \frac{-1}{9} & \frac{1}{45} & 0 & 0 & 0 & 0 & \cdots & 0\\
\vspace{.2cm}
\frac{1}{24} & \frac{-1}{9} & \frac{2}{9} & \frac{-1}{9} & \frac{1}{45} & 0 & 0 & 0 & \cdots & 0\\
\vspace{.2cm}
0 & \frac{1}{45} & \frac{-1}{9} & \frac{2}{9} & \frac{-1}{9} & \frac{1}{45} & 0 & 0 & \cdots & 0\\
\vspace{.2cm}
0 & 0 & \frac{1}{45} & \frac{-1}{9} & \frac{2}{9} & \frac{-1}{9} & \frac{1}{45}  & 0 & \cdots & 0\\
\vspace{.2cm}
\vdots & \vdots & \ddots & \ddots & \ddots  & \ddots & \ddots & \ddots & \ddots & \vdots \\
\vspace{.2cm}
0 & 0 & \cdots & 0  & \frac{1}{45} & \frac{-1}{9} & \frac{2}{9} & \frac{-1}{9} & \frac{1}{45}  & 0  \\
\vspace{.2cm}
0 & 0 & \cdots  & & 0  & \frac{1}{45} & \frac{-1}{9} & \frac{2}{9} & \frac{-1}{9} & \frac{1}{24}   \\
\vspace{.2cm}
0 & 0 & \cdots  &  & & 0  & \frac{1}{45} & \frac{-1}{9} & \frac{2}{9} & \frac{-1}{6}   \\
\vspace{.2cm}
0 & 0 & \cdots  & & & & 0  & \frac{1}{24} & \frac{-1}{6} & \frac{1}{4}   \\
\end{array}
\right).
\]
First notice that the events $A_i$ and $A_j$ are independent whenever $|i-j|>2$,
since each $A_i$ only depends on the ordering of $\s(i-1), \s(i)$ and $\s(i+1)$.
Thus, we get $\Cov({\bf 1}_{A_i}, {\bf 1}_{A_j})=0$ if $|i-j|>2$.
Remaining covariances on the diagonal strip $|i-j|\leq 2$ are computed as follows.
By symmetry assume $i\leq j$.
For $i=j$, one can easily have
\begin{align}
\Cov({\bf 1}_{A_i}, {\bf 1}_{A_i})=\Var({\bf 1}_{A_i})=P(A_i) (1-P(A_i))=\begin{cases} 1/4 & i=1  \text{ or\ } n-1,\\ 2/9 & 2\leq i \leq n-2. \end{cases} \label{eq:varAi}
\end{align}
Next, we compute the off diagonal terms on the strip $|i-j|\leq 2$.
For $i=1$, we have
\begin{align}
\Cov({\bf 1}_{A_1}, {\bf 1}_{A_2})&=P(A_1 \cap A_2)-\frac{1}{2} \cdot \frac{1}{3}
=P(\s(1)<\s(2), \s(2)<\s(1), \s(2)<\s(3))-\frac{1}{6} \nonumber\\
&= 0-\frac{1}{6}=\frac{-1}{6}, \label{eq:covA1A2}
\end{align}
since the event $\{ \s(1)<\s(2), \s(2)<\s(1), \s(2)<\s(3) \}$ can not occur
and
\begin{align}
\Cov({\bf 1}_{A_1}, {\bf 1}_{A_3})&=P(A_1 \cap A_3)-\frac{1}{2} \cdot \frac{1}{3}
=P(\s(1)<\s(2)> \s(3)<\s(4))-\frac{1}{6} \nonumber \\
&= \frac{5}{24}-\frac{1}{6}=\frac{1}{24}. \label{eq:covA1A3}
\end{align}
where $5/24$ comes from the fact that there are 5 alternating permutations of order 4.
By symmetry, we also have
\begin{align}
\Cov({\bf 1}_{A_{n-2}}, {\bf 1}_{A_{n-1}})=-1/6 \text{ and }
\Cov({\bf 1}_{A_{n-3}}, {\bf 1}_{A_{n-1}})=1/24. \label{eq:covAn-1}
\end{align}

When $j=i+1$, for each $2\leq i \leq n-3$ we have
\begin{align}
\Cov({\bf 1}_{A_i}, {\bf 1}_{A_{i+1}})&=\Cov({\bf 1}_{A_2}, {\bf 1}_{A_{3}}) =P(A_2 \cap A_{3})-\frac{1}{3} \cdot \frac{1}{3} \nonumber \\
&=P(\s(2)<\s(1), \s(2)<\s(3), \s(3)<\s(2), \s(3)<\s(4))-\frac{1}{9} = 0-\frac{1}{9}=\frac{-1}{9}, \label{eq:covA2A3}
\end{align}
since the event $\{ \s(2)<\s(1), \s(2)<\s(3), \s(3)<\s(2), \s(3)<\s(4) \}$ can not occur.
And, finally, when $j=i+2$, for each $2 \leq i \leq n-4$ we obtain
\begin{align}
\Cov({\bf 1}_{A_i}, {\bf 1}_{A_{i+2}})&=\Cov({\bf 1}_{A_2}, {\bf 1}_{A_{4}}) =P(A_2 \cap A_{4})-\frac{1}{3} \cdot \frac{1}{3} \nonumber \\
&=P(\s(1)> \s(2)<\s(3)> \s(4)<\s(5))-\frac{1}{9} \nonumber  \\
&= \frac{16}{120}-\frac{1}{9}=\frac{1}{45}. \label{eq:covA2A4}
\end{align}
where $16/120$ comes from the fact that there are 16 alternating permutations of order 5.
Therefore, for each $n\geq 5$, by combining the equations in \eqref{eq:varAi}-\eqref{eq:covA2A4} we obtain
\begin{align*}
\Var(R_n)&=\Var \left(\sum_{i=1}^{n-1} {\bf 1}_{A_i} \right)=\sum_{i,j=1}^{n-1} \Cov( {\bf 1}_{A_i}, {\bf 1}_{A_j}) \\
&=2 \cdot \frac{1}{4} +(n-3)\frac{2}{9} +4 \cdot \frac{-1}{6}+ 4 \cdot \frac{1}{24}+ 2(n-4)\frac{-1}{9} + 2(n-5)\frac{1}{45} \\
&= \frac{2n}{45}.
\end{align*}
\end{proof}

\subsection{Mean and Variance of $Q_n$}
The mean and variance of $Q_n$ can be derived in a similar manner.
\begin{thm}\label{thm:meanvarqn}
For a random sample of size $n$ from $U(0,1)$, the mean and the variance of the number of shared NNs, $Q_n$,
is $n/4$, for $n\geq 4$, and $19n/240$, for $n\geq 7$, respectively.
\end{thm}
\begin{proof}
Again by \eqref{eq:equiv}, we compute the mean and the variance of $\sum_{i=1}^{n-2} {\bf 1}_{B_i}$.
Clearly,
\[
\E({\bf 1}_{B_i})=P(B_i)=\begin{cases} 1/2 & i=1 \ {\text or\ } n-2,\\ 1/4 & 2\leq i \leq n-3, \end{cases}
\]
and hence, for every $n\geq 4$, we get
\[
\E(Q_n)=\E \left(\sum_{i=1}^{n-2} {\bf 1}_{B_i}\right)= \sum_{i=1}^{n-2} \E({\bf 1}_{B_i})=\frac{1}{2}+(n-4)\frac{1}{4} +\frac{1}{2}=\frac{n}{4}.
\]
For the variance, we compute the covariances of ${\bf 1}_{B_i}$'s in the following matrix:
\[
\left( \Cov({\bf 1}_{B_i}, {\bf 1}_{B_j}) \right)_{i,j =1}^{n-2}=
\left( \begin{array}{ccccccccccccc}
\vspace{.2cm}
\frac{1}{4} & 0 & \frac{-1}{8} & \frac{1}{24} & 0 & 0 & 0 & 0 & 0 & 0 & 0 & \cdots & 0\\
\vspace{.2cm}
0 & \frac{3}{16} & \frac{-1}{80} & \frac{-1}{16}& \frac{1}{48} & 0 & 0 & 0 & 0 & 0 & 0 & \cdots & 0 \\
\vspace{.2cm}
\frac{-1}{8} & \frac{-1}{80} & \frac{3}{16} & \frac{-1}{80} & \frac{-1}{16} & \frac{1}{48} & 0 & 0 & 0 & 0 & 0 & \cdots & 0 \\
\vspace{.2cm}
\frac{1}{24} & \frac{-1}{16} & \frac{-1}{80} & \frac{3}{16} & \frac{-1}{80} & \frac{-1}{16} & \frac{1}{48} & 0 & 0 & 0 & 0 & \cdots & 0\\
\vspace{.2cm}
0 & \frac{1}{48} & \frac{-1}{16} & \frac{-1}{80} & \frac{3}{16} & \frac{-1}{80} & \frac{-1}{16} & \frac{1}{48} & 0 & 0 & 0 & \cdots & 0 \\
\vspace{.2cm}
0 & 0 & \frac{1}{48} & \frac{-1}{16} & \frac{-1}{80} & \frac{3}{16} & \frac{-1}{80} & \frac{-1}{16} & \frac{1}{48} & 0 & 0 & \cdots & 0 \\
0 & 0 & 0 & \frac{1}{48} & \frac{-1}{16} & \frac{-1}{80} & \frac{3}{16} & \frac{-1}{80} & \frac{-1}{16} & \frac{1}{48} & 0 & \cdots & 0 \\
\vspace{.2cm}
\vdots & \vdots & \ddots & \ddots  & \ddots & \ddots & \ddots & \ddots  & \ddots & \ddots & \ddots & \ddots & \vdots \\
\vspace{.2cm}
0 & 0 &  \cdots &  & 0 & \frac{1}{48} & \frac{-1}{16} & \frac{-1}{80} & \frac{3}{16} & \frac{-1}{80} & \frac{-1}{16} & \frac{1}{48} & 0 \\
\vspace{.2cm}
0 & 0 & \cdots &  &  & 0 & \frac{1}{48} & \frac{-1}{16} & \frac{-1}{80} & \frac{3}{16} & \frac{-1}{80} & \frac{-1}{16}  & \frac{1}{24} \\
\vspace{.2cm}
0 & 0 & \cdots &  &  &  & 0 & \frac{1}{48} & \frac{-1}{16} & \frac{-1}{80} & \frac{3}{16} & \frac{-1}{80} & \frac{-1}{8} \\
\vspace{.2cm}
0 & 0 & \cdots &  &  &  &  & 0 & \frac{1}{48} & \frac{-1}{16} & \frac{-1}{80} & \frac{3}{16} & 0 \\
\vspace{.2cm}
 0 & 0 & \cdots &  &  &  &  &  & 0 & \frac{1}{24} & \frac{-1}{8} & 0 & \frac{1}{4}
\end{array} \right).
\]
Note that the events $B_i$ and $B_j$ are independent whenever $|i-j|>3$,
since each $B_i$ only depends on the ordering of $\s(i-1), \s(i), \s(i+1)$ and $\s(i+2)$.
Therefore, we have $\Cov({\bf 1}_{B_i}, {\bf 1}_{B_j})=0$ if $|i-j|>3$.
Remaining covariances (i.e., the entries in the diagonal strip $|i-j|\leq 3$) are computed as follows.
By symmetry suppose $i\leq j$.
When $i=j$, one can easily obtain the main diagonal terms
\begin{align}
\Cov({\bf 1}_{B_i}, {\bf 1}_{B_i})=\Var({\bf 1}_{B_i})=P(B_i) (1-P(B_i))=\begin{cases} 1/4 & i=1 \ \text{ or\ } n-2,\\ 3/16 & 2\leq i \leq n-3. \end{cases} \label{eq:varBi}
\end{align}
We next compute the off diagonal terms on the strip.
For $i=1$, we have
\begin{align}
\Cov({\bf 1}_{B_1}, {\bf 1}_{B_2})&=P(B_1 \cap B_2)-\frac{1}{2} \cdot \frac{1}{4}
=P(\s(2)<\s(3), \s(2)<\s(1), \s(3)<\s(4))-\frac{1}{8}  \nonumber \\
&= P(\s(2)<\min\{\s(1),\s(3),\s(4)\},  \s(3)<\s(4))-\frac{1}{8} = \frac{1}{4} \cdot \frac{1}{2}-\frac{1}{8}=0, \label{eq:covB1B2}
\end{align}
\begin{align}
\Cov({\bf 1}_{B_1}, {\bf 1}_{B_3}) &=P(B_1 \cap B_3)-\frac{1}{2} \cdot \frac{1}{4}
=P(\s(2)<\s(3), \s(3)<\s(2), \s(4)<\s(5))-\frac{1}{8} \nonumber \\
&=0-\frac{1}{8}=\frac{-1}{8}, \label{eq:covB1B3}
\end{align}
since the event $\{ \s(2)<\s(3), \s(3)<\s(2), \s(4)<\s(5) \}$ can not occur
and
\begin{align}
\Cov({\bf 1}_{B_1}, {\bf 1}_{B_4})&=P(B_1 \cap B_4)-\frac{1}{2} \cdot \frac{1}{4}
=P(\s(2)<\s(3), \s(4)<\s(3), \s(5)<\s(6))-\frac{1}{8}  \nonumber \\
&= P(\max\{\s(2),\s(4)\}<\s(3),  \s(5)<\s(6))-\frac{1}{8} = \frac{1}{3} \cdot \frac{1}{2}-\frac{1}{8}=\frac{1}{24}. \label{eq:covB1B4}
\end{align}
By symmetry, we get
\begin{align}
\Cov({\bf 1}_{B_{n-3}}, {\bf 1}_{B_{n-2}})=0, \Cov({\bf 1}_{B_{n-4}}, {\bf 1}_{B_{n-2}})=-1/8 \text{ and }
\Cov({\bf 1}_{B_{n-5}}, {\bf 1}_{B_{n-2}})=1/24. \label{eq:covBn-2}
\end{align}
When $j=i+1$ and $2\leq i \leq n-4$, we obtain
\begin{align}
\Cov({\bf 1}_{B_{i}}, {\bf 1}_{B_{i+1}})&=\Cov({\bf 1}_{B_{2}}, {\bf 1}_{B_{3}})=P(B_2 \cap B_{3})-\frac{1}{4} \cdot \frac{1}{4} \nonumber \\
&=P(\s(2)<\s(1), \s(3)<\s(4), \s(3)<\s(2), \s(4)<\s(5))-\frac{1}{16}  \nonumber \\
&=P(\s(3) < \min\{\s(1),\s(2), \s(4), \s(5)\}, \s(2)<\s(1), \s(4)<\s(5)) -\frac{1}{16} \nonumber \\
&=\frac{1}{5} \cdot \frac{1}{2} \cdot \frac{1}{2}-\frac{1}{16}=\frac{-1}{80}. \label{eq:covB2B3}
\end{align}
Similarly, if $j=i+2$ and $2\leq i \leq n-5$ we have
\begin{align}
\Cov({\bf 1}_{B_{i}}, {\bf 1}_{B_{i+2}})&=\Cov({\bf 1}_{B_{2}}, {\bf 1}_{B_{4}})=P(B_2 \cap B_{4})-\frac{1}{4} \cdot \frac{1}{4} \nonumber  \\
&=P(\s(2)<\s(1), \s(3)<\s(4), \s(4)<\s(3), \s(5)<\s(6))-\frac{1}{16} \nonumber \\
&=0-\frac{1}{16}=\frac{-1}{16}, \label{eq:covB2B4}
\end{align}
since the event $\{ \s(2)<\s(1), \s(3)<\s(4), \s(4)<\s(3), \s(5)<\s(6) \}$ can not occur.
Finally, for the case $j=i+3$ and $2\leq i \leq n-6$, we get
\begin{align}
\Cov({\bf 1}_{B_{i}}, {\bf 1}_{B_{i+3}})&=\Cov({\bf 1}_{B_{2}}, {\bf 1}_{B_{5}})=P(B_2 \cap B_{5})-\frac{1}{4} \cdot \frac{1}{4} \nonumber  \\
&=P(\s(2)<\s(1), \s(3)<\s(4), \s(5)<\s(4), \s(6)<\s(7))-\frac{1}{16} \nonumber \\
&=P(\max\{\s(3), \s(5)\}<\s(4), \s(2)<\s(1), \s(6)<\s(7)) -\frac{1}{16} \nonumber \\
&=\frac{1}{3} \cdot \frac{1}{2} \cdot \frac{1}{2}-\frac{1}{16}=\frac{1}{48}, \label{eq:covB2B5}
\end{align}
and therefore, for every $n\geq 7$, by combining the equations in \eqref{eq:varBi}-\eqref{eq:covB2B5} we obtain
\begin{align*}
\Var(Q_n)&=\Var \left(\sum_{i=1}^{n-2} {\bf 1}_{B_i} \right)=\sum_{i,j=1}^{n-2} \Cov( {\bf 1}_{B_i}, {\bf 1}_{B_j}) \\
&=2 \cdot \frac{1}{4} +(n-4)\frac{3}{16} +4 \cdot \frac{-1}{8}+ 4 \cdot \frac{1}{24}+ 2(n-5)\frac{-1}{80} + 2(n-6)\frac{-1}{16} + 2(n-7)\frac{1}{48} \\
&= \frac{19n}{240}.
\end{align*}
\end{proof}

\subsection{A Recurrence Relation for the Exact Distribution of $R_n$}
\label{sec:recursive-reln-Rn}
Recall that for a sequence of pairwise distinct real numbers $x_1,\dots ,x_n$, we say $x_k$ is a local minimum if $x_k$ is less than its neighbors (i.e., $x_{k-1}$ and $x_{k+1}$) for $2\leq k \leq n-1$.
Let us also consider $x_1$ (resp. $x_n$) as a local minimum if $x_1<x_2$ (resp. $x_n<x_{n-1}$).
Then, notice that $\sum_{i=1}^{n-1} {\bf 1}_{A_i}$ is exactly the number of local minima in a uniformly random permutation from $P_{n-1}$.
Let $p(n,k)$ denote $P(R_n=k)$ and set $p(1,0)=p(2,1)=1$.
Also, let $m(n,k)$ be the number of permutations in $P_n$ with exactly $k$ local minimums.
Notice that $p(n,k)=m(n-1,k)/(n-1)!$.
Since the term 1 in the sequence is always a local minimum and any two local minimums are not adjacent,
we see that $p(n,0)=0$ for $n\geq 2$ and $p(n,k)=0$ whenever $k>n/2$.

Any permutation in $P_n$ can be uniquely obtained by increasing each element of a permutation in $P_{n-1}$ by one and inserting the element 1 in one of the possible $n$ places.
In this process, inserting the element 1 into the sequence does not effect the number of local minimums if it is placed next to a local minimum, and otherwise, increases the number of local minimums by one.
Therefore, we obtain
\begin{align*}
m(n,k)=2k\cdot m(n-1,k)+(n-2(k-1))\cdot m(n-1,k-1),
\end{align*}
since any two local minimums are not adjacent.
Thus, as $m(n-1,k)=p(n,k)(n-1)!$, we have
\begin{align}
\label{eq:rec}
p(n+1,k)=\frac{2k}{n}  p(n,k)+\frac{n-2k+2}{n} p(n,k-1),
\end{align}
for every $n\geq 2$.
Therefore, the exact pmf of $R_n$ can be computed for any $n\geq 3$ by using the recursion given in \eqref{eq:rec}.

\cite{enns:1999} consider the index of the spacing with the largest length (i.e., index $i$ such that $D_i=\max \{D_1,D_2,\dots ,D_{n-1}\}$) and derive the following recurrence relation
\begin{align*}
p(n,k)= \frac{2}{n-1} p(n-1,k)+ \sum_{i=2}^{n-2} \sum_{j=1}^{k-1}\frac{ p(i,j) p(n-i,k-j)}{n-1},
\end{align*}
for every $n\geq 4$.
Then, using generating functions, they obtain the relation in \eqref{eq:rec}.

\section{Asymptotic Results for $R_n$ and $Q_n$}
\label{sec:asy-res}
In this section,
we prove SLLN results and CLTs for both $R_n$  and $Q_n$ as $n \rightarrow \infty$.
Observe that neither ${\bf 1}_{A_1}, {\bf 1}_{A_2}, \dots , {\bf 1}_{A_{n-1}}$ nor ${\bf 1}_{B_1}, {\bf 1}_{B_2}, \dots , {\bf 1}_{B_{n-2}}$ is an i.i.d. sequence.
However, both have a nice structure which allows SLLN and CLT results to follow.
\begin{definition}
A sequence of random variables $X_1,X_2, \dots , X_n$
is said to be \emph{$m$-dependent} if the random variables $(X_{1},X_2, \dots , X_{i})$ and $(X_{j}, X_{j+1}, \dots , X_{n})$
are independent whenever $j-i>m$.
\end{definition}
Since each $A_i$ (resp. $B_i$) only depends on the ordering of $\s(i-1), \s(i)$ and $\s(i+1)$ (resp. $\s(i-1), \s(i), \s(i+1)$ and $\s(i+2)$), it is clear to see that the sequence ${\bf 1}_{A_1}, {\bf 1}_{A_2}, \dots , {\bf 1}_{A_{n-1}}$
(resp. ${\bf 1}_{B_1}, {\bf 1}_{B_2}, \dots , {\bf 1}_{B_{n-2}}$) is 2-dependent (resp. 3-dependent).
For the asymptotic results,
note that we can ignore the random variables ${\bf 1}_{A_{1}}$ and ${\bf 1}_{A_{n-1}}$ (resp. ${\bf 1}_{B_{1}}$ and ${\bf 1}_{B_{n-2}}$),
since their contribution to the summand $\sum_{i=1}^{n-1} {\bf 1}_{A_{i}}$ (resp. $\sum_{i=1}^{n-2} {\bf 1}_{B_{i}}$) is negligible in the limit as $n$ goes to infinity.
Therefore, to obtain asymptotic results for $R_n$ (resp. $Q_n$) it suffices to consider $\sum_{i=2}^{n-2} {\bf 1}_{A_{i}}$ (resp. $\sum_{i=2}^{n-3} {\bf 1}_{B_{i}}$)
which is the sum of identically distributed  2-dependent (resp. 3-dependent)
indicator random variables with mean $1/3$ (resp. $1/4$).

For $m$-dependent identically distributed sequences, the SLLN extends in a straightforward manner by just partitioning the summand in $m+1$ sums of i.i.d. subsequences, and hence,
we obtain SLLN results for both $R_n$ and $Q_n$.
\begin{thm}\label{thm:slln}
{\rm (SLLN for $U(0,1)$ data)} For a random sample of size $n$ from $U(0,1)$, we have ${R_n}/{n} \xrightarrow{a.s.} {1}/{3}$ and ${Q_n}/{n} \xrightarrow{a.s.} {1}/{4}$ as $n\rightarrow \infty$, where $\xrightarrow{a.s.}$ denotes almost sure convergence.
\end{thm}
The asymptotic normality of the random variables we consider is due to the well-known results on the sequence of $m$-dependent identically distributed and bounded random variables (e.g., see \cite{hoeffding:1948}, \cite{chung:1974}).
\begin{thm}{\rm (CLT for $U(0,1)$ data)}\label{thm:cltrnqn}
For a random sample of size $n$ from $U(0,1)$, we have
$$
\frac{R_n-n/3}{\sqrt{2n/45}} \xrightarrow{\mathcal{L}} \mathcal{N}(0,1) \text{ and }
\frac{Q_n-n/4}{\sqrt{19n/240}} \xrightarrow{\mathcal{L}} \mathcal{N}(0,1),$$
as $n\rightarrow \infty$,
where $\xrightarrow{\mathcal{L}}$ denotes the convergence in law and $\mathcal{N}(0,1)$ is the standard normal distribution.
\end{thm}

\begin{remark}
\label{rem:u-stat}
{\bf Is $R_n$ or $Q_n$ a U-statistic?}
At first glance (a scaled form of) $R_n$ and $Q_n$ might look like a $U$-statistic of degree 2
with symmetric kernels as we can write them as
\begin{align}\label{equstat}
\frac{R_n}{ {n \choose 2}}=\frac{1}{ {n \choose 2}} \sum_{1\leq i< j \leq n} {\bf 1}_{A(i,j)}  \text{ and }
\frac{Q_n}{ {n \choose 2}}=\frac{1}{ {n \choose 2}} \sum_{1\leq i< j \leq n} {\bf 1}_{B(i,j)}
\end{align}
where $A(i,j)$ is the event that $\{X_i,X_j\}$ is a reflexive pair and
$B(i,j)$ is the event that $X_i$ and $X_j$ is share a NN.
If $R_n$ and $Q_n$ were $U$-statistics,
then asymptotic normality of both would follow by the general CLT for $U$-statistics (\cite{hoeffdingUstat:1948}).
However, the kernels (${\bf 1}_{A(i,j)}$ and ${\bf 1}_{B(i,j)}$) do not only depend on $X_i$ and $X_j$, but to all data points.
Hence the kernels are not of degree 2 but of degree $n$.
Although,
for $U$-statistics,
the degree can be equal to the sample size,
it should be a fixed quantity, $m$.
So $m \le n$ allows $m=n$, but this would be for small samples,
and as $n$ increases,
$m$ should stay fixed which is not the case here.
So, neither $R_n$ nor $Q_n$ is a $U$-statistic of finite (fixed) degree,
hence this approach would not work in proving the CLT for $R_n$ and $Q_n$.
\end{remark}

\begin{remark}\label{rem:rd}
{\bf Asymptotic Behavior of $R_n$ and $Q_n$ in Higher Dimensions.}
The results in \cite{henze:1987} and \cite{schiling:1986} imply that $\E(R_n)/n\rightarrow r(d)$ and $\E(Q_n)/n\rightarrow q(d)$  as $n\rightarrow \infty$,
where $r(d)$ and $q(d)$ are constants which only depend on the dimension $d$,
whenever the underlying distribution has an a.e. continuous density in $\mathbb{R}^d$
(i.e., $r(d)$  and $q(d)$ (somewhat unexpectedly) do not depend on the continuous distribution).
We have $r(1)=1/3$, $r(2)=3\pi /(8\pi +3\sqrt{3})\approx 0.3108$, $r(3)=8/27$, and in general,
\[
 r(d) =
  \begin{cases}
   \displaystyle \left[ 3+\sum_{k=1}^m \frac{1\cdot 3  \cdots   (2k-1)}{2\cdot 4 \cdots (2k)} \left(\frac{3}{4} \right)^k \right ]^{-1} & \text{if } d=2m+1, \\
   \displaystyle \left[ \frac{8}{3}+\frac{\sqrt{3}}{\pi} \left (1+\sum_{k=1}^{m-1} \frac{2\cdot 4  \cdots   (2k)}{3\cdot 5 \cdots  (2k+1)} \left(\frac{3}{4} \right)^k \right) \right ]^{-1} & \text{if } d=2m,
  \end{cases}
\]
(see, e.g., \cite{pickard:1982}).
On the other hand, the exact value of $q(d)$ is known only for $d=1$, $q(1)=0.25$.
For $d>1$, we only have empirical approximations, for example,
$q(2)\approx 0.315 $, $q(3) \approx 0.355$, $q(4)\approx 0.38$ and $q(5)\approx 0.4$
(\cite{schiling:1986}).
\end{remark}

\begin{remark}
Some other quantities based on NN digraph are of interest in the literature.
Notice that even though each point has a unique NN, it is not necessarily the NN of precisely one point.
Let $Q_{j,n}$ be the number of points in the data which are NN of exactly $j$ other points.
The quantities $Q_{j,n}$'s are used in tests for spatial symmetry (see, \cite{ceyhan:2014}).
Also in \cite{enns:1999}, $Q_{0,n}, Q_{1,n}$ and $Q_{2,n}$ correspond to the number of \emph{lonely}, \emph{normal} and \emph{popular} individuals in a population of size $n$, respectively.
Moreover, the fraction of points serving as NN to precisely $j$ other points (i.e., $Q_{j,n}/n$) is studied by many authors
(e.g., see \cite{clark:1955}, \cite{henze:1987} and \cite{newmanRT:1983}).
Clearly, in one dimension, a point is NN to at most two other points and hence $Q_{j,n}=0$ for every $j\geq 3$.
Double counting arguments for the number of vertices and the number of arcs give
$n=Q_{0,n}+Q_{1,n}+Q_{2,n}$ and $n=0\cdot Q_{0,n}+1\cdot Q_{1,n}+2\cdot Q_{2,n}$, respectively.
On the other hand, one can easily see $Q_n=\sum_{j\geq 0} {j \choose 2}Q_{j,n}=Q_{2,n}$ and obtain
$Q_n=Q_{0,n}=Q_{2,n}=(n-Q_{1,n})/2$.
Thus, for each $j=0,1,2$, we have SLLN and CLT results for $Q_{j,n}$ together with the exact values of its mean and variance
using the results on $Q_n$.
\end{remark}

\section{Discussion and Conclusions}
\label{sec:disc-conc}
In this article,
we study the probabilistic behavior of the number of reflexive nearest neighbors (NNs), denoted $R_n$,
and the number of shared NNs, denoted $Q_n$,
for one dimensional uniform data.
$R_n$ and $Q_n$ can also be viewed as graph invariants for the NN digraph with vertices being the data points,
and arcs being inserted from a point to its NN.
In particular,
we provide the means and variances of both $R_n$ and $Q_n$, and derive SLLN and CLT results for both of the quantities under the same settings.
We also present a recursive relation for the probability mass function (pmf)
of $R_n$,
which can provide the exact distribution of $R_n$ (by computation for finite $n$).
Recall that the results we obtain for $R_n$ (the mean, the variance and the recurrence relation)
are in agrement with the ones in \cite{enns:1999}.
However, our derivation of the results are different from theirs and their method is not applicable for $Q_n$.

This work lays the foundation for the study of (number) reflexive NN pairs, shared NN pairs
related invariants of NN digraphs in higher dimensions
which would be more challenging due to the lack of ordering of the data points in multiple dimensions.
Another potential research direction is that
the results can also be extended to data from non-uniform distributions in one or multiple dimensions.

Our Monte Carlo simulations suggest that CLT results for both $R_n$ and $Q_n$ seem to hold and
$\Var(R_n)/n$ and $\Var(Q_n)/n$ converge to $\sigma_R^2(d)$ and $\sigma_Q^2(d)$, respectively,
whenever the underlying process is a distribution on $\mathbb{R}^d$ with an a.e. continuous density,
where $\sigma_R^2(d)$ and $\sigma_Q^2(d)$ are constants which only depend on the dimension $d$.
Expectations are handled in \cite{henze:1987} and \cite{schiling:1986}, see Remark \ref{rem:rd}.
Notice that even in case of $d=1$,
we can not apply the method used in the paper when the distribution of the sample is not uniform since
we lose the interchangeability of the spacings.

The number of reflexive NN pairs was also used in inferential statistics in literature.
For example, \cite{dacey:1960} used it to test the clustering of river towns in US.
However, \cite{dacey:1960} ignored the uncertainty due to the randomness in $R_n$
and compared the observed $R_n$ values to its expected value
to declare clustering or regularity of the towns.
This methodology was also criticized by \cite{pinder:1975}
who proposed an alternative method based on the average NN distance (and its empirical pdf) for the same type of inference.
But,
$R_n$ can be employed in exact inference using its exact pmf for testing such one-dimensional patterns
for small $n$ (as the exact distribution depends on the distribution of the data).
However, by the above discussion on $r(d)$ and $\sigma_R^2(d)$,
for data from any continuous distribution,
$R_n$ would converge to the same normal distribution as $n$ goes to infinity.
Hence, testing spatial clustering/regularity based on the asymptotic approximation of $R_n$ is not appropriate
(hence not recommended),
as it would have power equal to the significance level of the test in the limit
under any continuous alternative as well as under the null pattern (i.e., under uniformity of the points).
On the other hand,
if the convergence in probability of $R_n/n$ and $Q_n/n$ to some constants (regardless of the distribution of continuous data)
is established in all dimensions,
then this would be a desirable property for removing the restrictions of NN tests
which are conditional on $R_n$ and $Q_n$, (e.g., tests of \cite{dixon:1994})
in their asymptotic distribution.
The types of convergence for $R_n/n$ and $Q_n/n$ for data in higher dimensions are topics of ongoing research.

\section*{Acknowledgments}
EC was supported by the European Commission under the Marie Curie International Outgoing Fellowship Programme
via Project \# 329370 titled PRinHDD.

\end{document}